\documentclass[12pt]{amsart}
\usepackage{tikz}
 \usetikzlibrary{graphs}
\usepackage{graphicx}
 \usepackage{amsmath}
 \usepackage{amssymb}
\usepackage{amsthm}
\usepackage{amsfonts}
\usepackage{dsfont}
\usepackage[OT4]{fontenc}
\usepackage{pst-node}
\usepackage{pst-plot}
\usepackage{pst-tree}
\newtheorem{lemma}{\bf Lemma}
\newtheorem{prop}{\bf Proposition}
\newtheorem{cor}{\bf Corollary}
\newtheorem{thm}{\bf Theorem}
\usepackage[cp1250]{inputenc}
\usepackage{srcltx}

\long\def\comment#1{{}}
\begin{document}
\noindent 

\title{Inhomogeneous Jacobi matrices on trees}
\author{Ryszard Szwarc}
\date{}

\address{R.~Szwarc, Institute of Mathematics, University of Wroc\l aw, pl.\ Grunwaldzki 2/4, 50-384 Wroc\l aw, Poland}
\email{szwarc@math.uni.wroc.pl}
\thanks{The author acknowledges support by Polish grant NCN 2013/11/B/ST1/02308}


\begin{abstract}
We study  Jacobi matrices on trees  with one end at inifinity. We show that the defect indices cannot be greater than 1 and give criteria for essential selfadjointness.  We construct certain polynomials associated with matrices which mimic orthogonal polynomials in the classical case. Nonnegativity of Jacobi matrices is studied as well.
\end{abstract}
\maketitle
\section{Introduction}
The aim of the paper is to study  a special class of symmetric unbounded operators and their spectral properties. These are Jacobi operators defined on one sided  tree. They are immediate generalizations of classical Jacobi matrices which act on sequences $\{u(n)\}_{n=0}^\infty$ by the rule
$$ Ju(n) =\lambda_nu(n+1) +\beta_nu(n) +\lambda_{n-1}u(n-1),\quad n\ge 0,$$ where
$\{\lambda_n\}_{n=0}^\infty$ and $\{\beta_n\}_{n=0}^\infty$ are sequences of positive and real numbers, respectively, with the convention $u(-1)=\lambda_{-1}=0.$
These matrices are closely related with the set of polynomials defined recursively by
\begin{equation}xp_n(x) =\lambda_np_{n+1}(x) +\beta_np_n(x) +\lambda_{n-1}p_{n-1}(x),\quad n\ge 0, 
\end{equation}
with $p_{-1}=0, \ p_0=1.$

In case the coefficients of the matrix are bounded, the matrix $J$ represents selfadjoint operator on $\ell^2(\mathbb{N}_0).$ If $E(x)$ denotes the resolution of identity associated to $J,$ then the polynomials $p_n(x)$ are orthonormal with respect to the measure
$d\mu(x)=d(E(x)\delta_0,\delta_0),$ where $\delta_0$ is the sequence taking value 1 at $n=0$ and vanishing elsewhere and $(u,v)$ denotes the standard inner product in $\ell^2(\mathbb{N}_0).$ The measure $\mu$ has bounded support. 

When the coefficients are unbounded the operator $J$ is well defined on the domain $D(J)$ consisting of sequences with finitely many nonzero terms. In that case,  if this operator is essentially selfadjoint 
then again the polynomials $p_n$ are orthonormal with respect to the measure
$d\mu(x)=d(E(x)\delta_0,\delta_0),$ except that this measure has unbounded support.
Moreover there is a unique orthogonality measure  for polynomials $p_n.$
By a classical theorem, if  the operator $J$ is not essentially selfadjoint, there
are many measures $\mu$ on the real line so that 
the polynomials belong to $L^2(\mu),$ i.e.
$$ \int\limits_{-\infty}^{\infty}x^{2n}\,d\mu(x)<\infty $$ and the polynomials $p_n$ are orthogonal with respect to the inner product $$(f,g)=\int\limits_{-\infty}^\infty f(x)\overline{g(x)}\,d\mu (x).$$

Therefore essential selfadjointness is a crucial property that distinguishes between the so called determinate and indeterminate cases. Intuitively the unbounded matrix $J$ is essentially selfadjoint when the coefficients have moderate growth. But the converse is not true in general. For the classical theory of Jacobi matrices, orthogonal polynomials and moment problems we address the reader to \cite{akh}, \cite{ch}, \cite{st}, and to \cite{simon} for a modern treatment.

In a recent paper \cite{ksz}  homogeneous Jacobi matrices on one sided homogeneous trees were studied.  Two types of homogeneous trees were considered. One of them was the tree with infinitely many origin points (on level 0) and one end at infinity. 
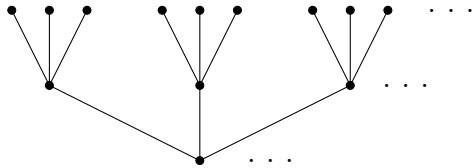
\vspace{0.5cm}
\begin{center}
\begin{tikzpicture}
\draw[fill=black] (-2,1) circle (1.5pt);
\draw[fill=black] (0,1) circle (1.5pt);
\draw[fill=black] (2,1) circle (1.5pt);
\draw[fill=black] (0,0) circle (1.5pt);
\draw (0,0)--(-2,1);
\draw (0,0)--(-0,1);
\draw (0,0)--(2,1);
\draw[fill=black] (-2.5,2) circle (1.5pt);
\draw[fill=black] (-2,2) circle (1.5pt);
\draw[fill=black] (-1.5,2) circle (1.5pt);
\draw (-2.5,2)--(-2,1);
\draw (-2,2)--(-2,1);
\draw (-1.5,2)--(-2,1);
\draw[fill=black] (-0.5,2) circle (1.5pt);
\draw[fill=black] (0,2) circle (1.5pt);
\draw[fill=black] (0.5,2) circle (1.5pt);
\draw (-0.5,2)--(0,1);
\draw (0,2)--(0,1);
\draw (0.5,2)--(0,1);
\draw[fill=black] (1.5,2) circle (1.5pt);
\draw[fill=black] (2,2) circle (1.5pt);
\draw[fill=black] (2.5,2) circle (1.5pt);
\draw (1.5,2)--(2,1);
\draw (2,2)--(2,1);
\draw (2.5,2)--(2,1);
\node at (3.2,2) {\ \ .\ .\ .};
\node at (2.6,1) {\ \ .\ .\ .};
\node at (0.8,0) {\ \ .\ .\ .};
\end{tikzpicture}
\end{center}

The tree $\Gamma$ consists of vertices on levels from zero to infinity. Every vertex $x$ on level $n$ is connected with a unique vertex $x'$ on level $n+1$ and $d$ vertices $x_1,\ldots, x_d$ on level $n-1$ for $n\ge 1$ like in the figure below:

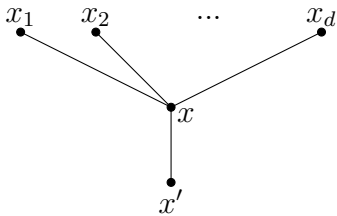
\begin{center}
\begin{tikzpicture}
\draw[fill=black] (-1,2) circle (1.5pt);
\draw[fill=black] (1,1) circle (1.5pt);
\draw[fill=black] (0,2) circle (1.5pt);
\draw[fill=black] (3,2) circle (1.5pt);
\draw[fill=black] (1,0) circle (1.5pt);
\node at (1.2,0.9) {$x$};
\node at (1,-0.25) {$x'$};
\node at (-1,2.2) {$x_1$};
\node at (0,2.2) {$x_2$};
\node at (3,2.2) {$x_d$};
\node at (1.5,2.2) {$...$};
 \draw (-1,2)--(1,1);
 \draw (0,2)--(1,1);
 \draw (3,2)--(1,1);
 \draw (1,0)--(1,1);
\end{tikzpicture}
\end{center}

The Jacobi matrices were defined on $\ell^2(\Gamma )$ where $\Gamma$ denotes the set of all vertices of the tree. 
The formula is as follows (when the degree equals $d+1$)
$$ Jv(x) =\lambda_n v(x') +\beta_n v(x) +\lambda_{n-1} [v(x_1)+v(x_2)+\ldots +v(x_d)],$$
where $n$ denotes the level of the vertex $x$ counting from above.

An interesting phenomenon occured. It turned out that the  operator $J$ defined on
functions $\{v(x)\}_{x\in \Gamma},$ with finitely many nonzero terms, is always essentially selfadjoint, regardless of the growth of the coefficients $\lambda_n$ and $\beta_n.$ 
For example the operator $J$ with coefficients $\lambda_n=(n+1)^2$ and $\beta_n=0$ is not essentially selfadjoint when considered as the classical Jacobi matrix on $\ell^2(\mathbb{N}_0).$ But it is essentialy selfadjoint when it acts on $\ell^2(\Gamma).$

Moreover its spectrum is discrete and consists of the zeros of all the polynomials $p_n$ associated with classical Jacobi matrix with coefficients $\sqrt{d}\,\lambda_n$ and $\beta_n,$ i.e. satisfying
$$xp_n(x) =\sqrt d\,\lambda_np_{n+1}(x) +\beta_np_n(x) +\sqrt d\,\lambda_{n-1}p_{n-1}(x),\quad n\ge 0.$$

Our aim is to study inhomogeneous Jacobi matrix on that tree. This means we do not require that the coefficients of the matrix are constant on each level of the tree. With every vertex $x$ we associate a positive number $\lambda_x$ and a real number $\beta_x.$  We are going to study operators of the form 
$$ Jv(x)=\lambda_{x}v(x')+\beta_xv(x) +\lambda_{x_1}v(x_1)+
\lambda_{x_2}v(x_2)+\ldots +\lambda_{x_d}v(x_d).$$

One of the main differences between the classical case and the case of the tree $\Gamma$ is that the eigenvalue equation
\begin{equation}zv(x)=\lambda_x v(x')+\beta_x v(x)+\lambda_{x_1}v(x_1)+
\lambda_{x_2}v(x_2)+\ldots +\lambda_{x_d}v(x_d)
\end{equation}
 cannot be solved recursively, unlike the equation
$$zv(n)=\lambda_nv(n+1)+\beta_nv(n)+\lambda_{n-1}v(n-1).$$
This not a coincidence as we are going to show that the equation (2) may not admit nonzero solutions for real values of $z$ (cf.  Proposition 5).  But we will show the equation has a nonzero solution for every nonreal $z$ (Corollary 3).

Actually, when we give up homogeneity of the matrix $J,$ we can as well give up homogoneity of the tree. Again $x'$ is the only vertex one level up connected with $x$ and $x_1,x_2,..., x_d$ are all vertices one level down connected with $x.$ The number $d$ may now vary if we do not assume homogeneity of the tree.

The  operator $J$ is  symmetric on $\ell^2(\Gamma)$ with respect to the natural inner product
$$ (u,v)=\sum_{x\in \Gamma} u(x)\overline{v(x)}.$$
We are interested in studying the essential selfadjointness of the matrix $J.$ It turns out that unlike in homogeneous case, the matrix $J$ may not be essentially selfadjoint. However the defect indices cannot be greater than 1 (Corollary 3). We derive certain criteria assuring essential selfadjointness. For example the analog of Carleman condition holds (see Theorem 6). Moreover we relate essential selfadjointness of $J$  with essential selfadjointness of the classical Jacobi matrix $J_0$ obtained from $J$ by restriction to an infinite path of the tree (see Theorem 4 and Remark following it). 

Classical Jacobi matrices are associated with orthogonal polynomials through the formula (1). In case of the tree $\Gamma$ there is no natural way of defining polynomials associated with Jacobi matrices on $\Gamma,$ since (as was mentioned above) the eignevalue equation may be not solvable. In Section 3 we define certainpolynomials associated with $J.$   We prove that they have real and simple zeros. Also we show  interlacing property for roots of two consecutive polynomials. We also prove that the zeros of these polynomials describe the spectrum of restriction of $J$ to finite subtrees of $\Gamma.$ However, unlike in the classical case, there is no natural orthogonality relation between these polynomials.

In Section 5 we give a criterion for nonnegativity of the Jacobi matrix $J$ on $\Gamma.$ In the classical case the Jacobi matrix $J$ is positive definite if and only if $(-1)^np_n(0)>0$ for every $n ,$ where $p_n$ are the orthogonal polynomials associated with $J.$  In case of  tree $\Gamma$ we do not have solutions of eigenvalue problem at our disposal or orthogonal polynomials. Therefore we had to find another way of getting the result. The nonnegativity of the matrix $J$ proved to be a useful tool in construction of a Jacobi matrix on $\Gamma$ for which the eigenvalue equation (2) does not admit solutions for some real values.

\section{Definitions and basic properties}
We will consider a tree $\Gamma$ with one end at infinity. Its vertices are located on levels from zero to infinity.  Every vertex $x$ on level $k\ge 0$ is connected with a unique vertex $x'$ on level $k+1.$ Moreover, when $k\ge 1$ the vertex $x$ is connected  with a finite number of vertices $y$ on level $k-1.$ This set will be denoted by $N_x.$ Let $\ell(x)$ denote the level of vertex $x.$

For a given vertex $x$ let $\Gamma_x$ denote the finite subtree containing the vertex $x$ together with vertices  $y$ so that $\ell(y)<\ell(x)$ and connected with $x$ by a path.  Denote $\Gamma_x'=\Gamma_x\cup\{x'\}.$ 

 Define $\mathcal{F}(\Gamma)$ to be the set of all complex valued functions with finite support on $\Gamma.$ Let 
 $$\delta_x(y)=\begin{cases}
 1 & y=x,\\
 0 & y\neq x.
 \end{cases}$$ 
 Consider the operator $J$  acting on $\mathcal{F}(\Gamma)$ according to the rule
\begin{align}
 &J\delta_x=\lambda_{x}\delta_{x'} +\beta_{x}\delta_x+\sum_{y\in N_{x}}\lambda_y\delta_y,& \ell(x)> 1,\\
 &J\delta_x=\lambda_{x}\delta_{x'} +\beta_{x}\delta_x,&\ell(x)=0,
\end{align}
 where
$\lambda_x$ are positive constants while $\beta_x$ are real ones.
Let $S$ be the operator acting by the rule
$$S\delta_x=\lambda_x \delta_{x'}.$$ Then the adjoint operator $S^*$ is given by
$$S^*\delta_x=\begin{cases}\displaystyle\sum_{y\in N_{x}}\lambda_y\delta_y,& l(x)>0,\\
S^*\delta_x=0, &l(x)=0.
\end{cases}
$$
Let $M$ be a  multliplication operator defined by
$$M\delta_x=\beta_x\delta_x.$$
Then 
\begin{equation}\label{s}
J=S+S^*+M.
\end{equation}
In particular $J$ is a symmetric linear operator.

We will study formal eignefunctions of the operator $J,$ i.e. functions $v$ defined on $\Gamma$ satisfying
$$ Jv=zv.$$ Evaluation at the vertex $x$ gives that equivalently we have the recurrence relation
\begin{equation}\label{recur}
zv(x)=\begin{cases}
\lambda_xv(x')+\beta_xv(x)+\displaystyle\sum_{y\in N_x}\lambda_yv(y) & \ell(x)\ge 1,\\
\lambda_xv(x')+\beta_xv(x) &\ell(x)=0.
\end{cases}
\end{equation}
In order to simplify the notation set $N_x=\emptyset$ if $\ell(x)=0.$ Then
(\ref{recur}) takes the form
\begin{equation}\label{recurc}
zv(x)=
\lambda_xv(x')+\beta_xv(x)+\displaystyle\sum_{y\in N_x}\lambda_yv(y)
\end{equation}
\begin{lemma}
Fix a vertex $x\in \Gamma.$ 
Assume there exists $0\neq v\in \mathcal{F}(\Gamma_{x'})$ and $z\notin \mathbb{R}$ such that $Jv(y)=zv(y)$ for
$y\in \Gamma_{x}.$ Then $v(x')\neq 0.$
\end{lemma}
\begin{proof}
Assume for a contradiction that $v(x')=0.$ Let $J_x$ denote the finite Jacobi matrix defined on $\Gamma_{x}$ obtained from $J$ by
truncation. Namely $J_x=P_xJP_x,$ where $P_x$ denotes the orthogonal projection from $\ell^2(\Gamma)$ to $\ell^2(\Gamma_x).$ Let $w$ denote the truncation of $v$ to $\Gamma_{x}.$ Since $v(x')=0$ we have
$J_xw=zw.$ Moreover $w\neq 0.$ Therefore $z$ must be a real number, as $J_x$ is a finite dimensional symmetric linear operator.
\end{proof}
\begin{lemma} Fix a vertex $x\in \Gamma.$ 
Assume there exists $0\neq v\in \mathcal{F}(\Gamma_x)$ and $z\notin \mathbb{R}$ such that  $Jv(y)=zv(y)$ for
$y\in \Gamma_{x}\setminus\{x\}.$ Then
$$zv(x)\neq \beta_{x}v(x) +\sum_{y\in N_{x}}\lambda_y v(y).$$
\end{lemma}
\begin{proof}
Assume for a contradiction that
$$zv(x)= \beta_{x}v(x) +\sum_{y\in N_{x}}\lambda_y v(y).$$
Define the function $u\in \mathcal{F}(\Gamma_x')$ by setting $u(y)=v(y)$ for $y\in \Gamma_x$ and $u(x')=0.$ Then
$Ju(y)=zu(y)$ for $y\in \Gamma_x.$ In view of Lemma 1 we get a contradiction.
\end{proof}
\begin{cor}
Assume there exists a function $v\neq 0$ on 
$\Gamma$ and $z\notin \mathbb{R}$ such that $Jv(x)=zv(x)$ for $x\in \Gamma.$ Then  $v$ does not vanish on $\Gamma.$
\end{cor}
\begin{proof}
Assume for a contradiction that $v(x)=0$ for a vertex $x.$ By Lemma 1 we get that
 the function $v$ vanishes identically on $\Gamma_{x}.$ From the recurrence relation 
$$zv(x)=\lambda_{x}v(x') +\beta_{x}v(x)+\sum_{y\in N_{x}}\lambda_yv(y),$$
we get $v(x')=0.$ Therefore $v$ vanishes identically
on $\Gamma_{x'}.$ Applying the same procedure infinitely many times we achieve that $v$ vanishes at every vertex of $\Gamma.$
\end{proof}

\begin{lemma} For any nonreal number $z$ and any $x_0\in \Gamma$ with $l(x_0)\ge 1$ there exists a nonzero function $v$ defined on $\Gamma_{x_0}$ satisfying
\begin{equation}\label{eq} Jv(x)=zv(x),\quad x\in \Gamma_{x_0}\setminus\{x_0\}.
\end{equation}
Moreover the function $v$ cannot vanish and is unique up to a constant multiple.
\end{lemma}
\begin{proof}
We will prove Lemma 3 by induction on the level $l(x_0).$ Assume $l(x_0)=1.$ Set $v(x_0)=1.$
Let $x\in N_{x_0}.$ Then $l(x)=0.$ We want to have
$$zv(x)=\lambda_xv(x_0)+\beta_xv(x).$$
Thus we may set
$$v(x)={\lambda_xv(x_0)\over z-\beta_x}.$$ In this way (\ref{eq}) is fulfilled.

Assume the conclusion is true for all vertices on the level $n. $ Let $l(x_0)=n+1.$
Consider vertices $x_1,x_2,\ldots, x_k\in N_{x_0}.$ Then $l(x_i)=n$ for $i=1,2,\ldots, k.$ By induction hypothesie, for every vertex $x_i$ there exists a nonzero function
$v_i$ defined on $\Gamma_{x_i}$ satisfying
$$Jv_i(x)=zv_i(x),\quad x\in \Gamma_{x_i}\setminus\{x_i\}.$$ We have
$$ \Gamma_{x_0}=\bigcup_{i=1}^k \Gamma_{x_i}\cup \{x_0\}.$$
We are going to define the function $v$ on $\Gamma_{x_0}$ in the following way: set 
$$ v(x)=c_iv_i(x), \quad {\rm for}\ x\in \Gamma_{x_i},$$ with $c_1=1.$
In this way we get  $$Jv(x)=zv(x),\quad x\in \Gamma_{x_i}\setminus \{x_i\},\quad i=1,2,\ldots, k.$$
In order to conclude the proof we must show that 
$$Jv(x_i)=zv(x_i),\quad i=1,2,\ldots, k.$$ 
Thus we want to have
$$ zc_iv_i(x_i)= \lambda_{x_i}v(x_0)+\beta_{x_i}c_iv_i(x_i) +\sum_{y\in N_{x_i}}\lambda_y c_iv_i(y),$$
i.e.
\begin{equation}\label{nonzero1}
\lambda_{x_i}v(x_0)=c_i\left (zv_i(x_i)-\beta_{x_i}v_i(x_i) -\sum_{y\in N_{x_i}}\lambda_y v_i(y)\right ).
\end{equation}
The expression in the brackets on the right hand side is nonzero for every $i=1,2,\ldots, k$ by Lemma 2. Therefore  (\ref{nonzero1}) is satisfied
for an appropriate choice of the value $v(x_0)$ and   nonzero constants $c_2,c_3,\ldots, c_k.$

By Lemma 1 the function $v$ cannot vanish. Moreover if  there was another function $\tilde{v}$ satisfying the conclusion of Lemma 3,
then $v-c\tilde{v}$ would also satisfy the conclusion and would vanish for an approprate choice of the constant $c.$ Thus $v=c\tilde{v}.$
\end{proof}
By the proof of Lemma 3 we get the following.
\begin{cor} For any nonreal number $z$ there exists a nonzero function $v$ so that
$$Jv(x)=zv(x),\quad x\in \Gamma.$$ The function $v$ cannot vanish and is unique up to a constant multiple.
\end{cor}

{\bf Remark.} In Chapter 5 we are going to show that the conclusion of Corollary 2 may not be true for real numbers $z.$ Observe that for classical Jacobi matrices (when $\Gamma=\mathbb{N}_0$)  the recurrence relation
\begin{equation}\label{orth}zv(n)=
\lambda_nv(n+1)+\beta_nv(n)+\lambda_{n-1}v(n-1)
\end{equation}
($\lambda_{-1}=0$) 
has nonzero solutions for any $z\in \mathbb{C}.$
\section{Polynomials and zeros}
The classical Jacobi matrices are related to orthogonal polynomials. Namely setting $v_0=1$ in (\ref{orth}) gives that $v(n)=p_n(z),$ where $p_n$ is a polynomial of order $n,$ with real coefficients.  The question arises if Jacobi matrices on trees are connected to polynomials, as well. In general we cannot expect that the solution of $Jv=zv$ will satisfy that $v(t)=P_t(z),$ where $P_t$ is a polynomial for every $t\in \Gamma.$ But we may expect that $P_t(z)$ is a polynomial for $t$ in a subtree $\Gamma_x$ for some $x\in \Gamma.$
\begin{prop} Let $x\in \Gamma.$ There exists a nonzero solution $v_x$ of $Jv_x=zv_x,$ so that for any $t\in \Gamma'_x$ the function
$v_x(t)=P_{x,t}(z)$ is a polynomial with real coefficients and positive leading coefficient.  Moreover
if $t\in \Gamma'_{y}\subset\Gamma'_x$ then the polynomial $P_{x,t}$ is divisible by $P_{{y},t}.$
\end{prop}
\begin{proof}
We will use induction on the level $l(x).$ Let $l(x)=0.$ By Corollary 2 there is a nonzero solution $v$ of $Jv=zv.$ Then $v(x)\neq 0.$  
Let
$$v_x={v(x)}^{-1}v.$$ 
Hence $v_x(x)=1.$ By $Jv_x=zv_x$ evaluated at $x$ we get 
$$v_x(x')={z-\beta_x\over \lambda_x}.$$

Assume now the conclusion is valid for vertices on the level $n.$ Let $l(x)=n+1.$
By induction hypothesis, for any $y\in N_x$ there is a nonzero solution $v_y$ so  that
$P_{y,t}(z)$ is a polynomial with real coefficients for $t\in \Gamma'_y.$ 
In particular the polynomial $v_y(x)=P_{y,x}(z)$ has real coefficients.
Moreover by Lemma 1 the polynomial $P_{y,x}(z)$ cannot vanish for $z\notin \mathbb{R}.$
Fix $y_1\in N_x$ and let
\begin{equation}v_x={{\rm LCM}\{P_{y,x}(z)\,:\,y\in N_x\}\over P_{y_1,x}(z)} \ v_{y_1}.\label{v}
\end{equation}
Since $Jv_{y_1}=zv_{y_1}$ we get $Jv_x=zv_x.$
Moreover $v_x$ does not vanish for $z\notin \mathbb{R}.$
Since $v_{y_1}(x)=P_{y_1,x}(z)$ we obtain
\begin{equation}P_{x,x}(z)=v_x(x)={\rm LCM}\{P_{y',x}(z)\,:\,y'\in N_x\}.\label{vx}
\end{equation} 
Since the value $v_x(x)$ determines the solution, the function $v_x$ does not depend on the choice of $y_1\in N_x.$
Thus the formula (\ref{v}) and the above reasoning is valid for any choice of  $y\in N_x.$
Hence
\begin{equation}v_x={{\rm LCM}\{P_{\tilde{y},x}(z)\,:\,\tilde{y}\in N_x\}\over P_{{y},x}(z)} \ v_{{y}},\qquad {y}\in N_x.\label{vv}
\end{equation}
 By (\ref{vv}) and by induction hypothesis the value $v_x(t)$ is a polynomial in $z$ for any $t\in\displaystyle \bigcup_{y\in N_x}\Gamma_y\cup \{x\}=\Gamma_x.$ By the recurrence relation also the value $v_x(x')$ is a polynomial. Moreover by (\ref{vv})
the polynomial $v_x(t)$ is divisible by $v_y(t)$ for any $y\in N_x$ and $t\in \Gamma'_y.$ This implies the last part of the conclusion.  
\end{proof}
\noindent{\bf Remarks} 
The formulas (\ref{vx}) and (\ref{vv}) imply that for $y\in N_x$ and $t\in \Gamma_y'$ we have
\begin{equation}\label{vxx}
P_{x,t}(z)=P_{x,x}(z){P_{y,t}(z)\over P_{y,x}(z)}.
\end{equation} 
 Let $y\in \Gamma_x.$ Then $y$ and $x$ are connected in $\Gamma_x$ by a path
$y=y_0,y_1,\ldots, y_n=x.$ By iterating (\ref{vxx})  we get
$$ P_{x,y}(z)
={P_{y_n,y_n}(z)\over P_{y_{n-1},y_n}(z)}\cdot {P_{y_{n-1},y_{n-1}}(z)\over P_{y_{n-2},y_{n-1}}(z)}\cdot \ldots\cdot  {P_{y_{1},y_1}(z)\over P_{y_{0},y_{1}}(z)}P_{y_0,y_0}(z).
$$
Let $y\in \Gamma_{\tilde{x}}\subset \Gamma_x.$ Then $\tilde{x}=y_k$ for some $k,$ $0\le k\le n.$ Hence
$$P_{x,y}(z)={P_{y_n,y_n}(z)\over P_{y_{n-1},y_n}(z)}\cdot {P_{y_{n-1},y_{n-1}}(z)\over P_{y_{n-2},y_{n-1}}(z)}\cdot \ldots\cdot  {P_{y_{k+1}}(z)\over P_{y_{k},y_{k+1}}(z)} \,P_{\tilde{x},y}(z).$$
These formulas and (\ref{vx}) imply that the polynomial $P_{x,y}(z)$ can be described in terms of the polynomials of the form $P_{t,t'}(z)$ for $t\in \Gamma_x.$ 
\begin{cor} Let $z\notin \mathbb{R}.$ Let $\{x_n\}_{n=0}^\infty$ be an infinite path in $\Gamma $ so that $l(x_n)=n.$
Let $v$ be a nonzero solution of $Jv(x)=zv(x)$ so that  $v(x_0)=1.$   Then
for any vertex $x\in \Gamma_{x_n}$ we have $$v(x)={a_{n,x}(z)\over b_{n}(z)},$$  
 where $a_{n,x}(z)$ and $b_{n}(z)$ are polynomials with real coefficients.
Moreover  the polynomial $b_{n+1}$ is divisible by $b_{n}.$ 
\end{cor}
\begin{proof}
Consider the subtree $\Gamma_{x_n}.$ Let $x\in \Gamma_{x_n}.$ By Proposition 1 there is a solution $v_n$ so that 
$v_n(x)$ and $v_n(x_0)$ are polynomials with real coefficients. Then 
$$v(x)={v_n(x)\over v_n(x_0)}$$  satisfies $v(x_0)=1.$ By the last part of Proposition 1 the polynomial $v_{n+1}(x_0)$ is divisible by $v_n(x_0).$
\end{proof}
\begin{thm}
The polynomials $P_{x,y}(z),$ $y\in \Gamma_x',$  have only real zeros. Moreover 
for any $x\in \Gamma$ 
the zeros of $P_{x,x}$ and $P_{x,x'}$ are single, and
the zeros of $P_{x,x}$ interlace with the zeros of $P_{x,x'},$ i.e. if $x_1<x_2<\ldots< x_n$  denote the zeros of $P_{x,x'},$ then $P_{x,x}$ has $n-1$ zeros $y_1<y_2<\ldots <y_{n-1}$ and
$$x_1<y_1<x_2<y_2<\ldots < y_{n-1}<x_n.$$
\end{thm}
\begin{proof}
We will use induction on $l(x).$ Let $l(x)=0.$ Then $P_{x,x}=1$ and
$P_{x,x'}={(z-\beta_x)/ \lambda_x}.$ Assume the conclusion is valid for $l(x)=n-1.$ 
Let $l(x)=n.$ By the recurrence relation we have
\begin{equation}\label{rec1}\lambda_xP_{x,x'}(z)=(z-\beta_x)P_{x,x}(z)-\sum_{j=1}^k\lambda_{y_j}P_{x,y_j}(z),
\end{equation} where $N_x=\{y_1,y_2,\ldots,y_k\}.$ By (\ref{vxx}), with $t=y_j,$  we get
\begin{equation}\label{deg}P_{x,y_j}(z)=P_{x,x}(z)\,{P_{y_j,y_j}(z)\over P_{y_j,x}(z)}.
\end{equation}
By induction hypothesis the zeros of $P_{y_j,y_j}(z)$ are real and single and interlace with the zeros of $P_{y_j,x}(z)$ for any $j.$ This implies
$$\deg P_{y_j,x}=\deg P_{y_j,y_j}+1. $$
In view of (\ref{rec1}) and (\ref{deg}) we get
$$\deg P_{x,x'}=\deg P_{x,x}+1.$$
 Let $r$ be a  root of $P_{x,x}(z).$ We are going to study the sign of 
$P_{x,x'}(r)$
making use of (\ref{rec1}). 
If $P_{y_j,x}(r)\neq 0,$ then (\ref{deg}) implies $P_{x,y_j}(r)=0.$  
But since $P_{x,x}(r)=0$ then $P_{y_{j_0},x}(r_1)=0$ for some $j_0,$ by (\ref{vx}). Consider the quantity
$$P_{x,y_{j_0}}(r+\varepsilon)=P_{x,x}(r+\varepsilon)  \,{P_{y_{j_0},y_{j_0}}(r+
\varepsilon)\over P_{y_{j_0},x}(r+\varepsilon)},$$ where $\varepsilon>0$ is infinitesimally small. We have
$$ {P_{y_{j_0},y_{j}}(r+
\varepsilon)\over P_{y_{j_0},x}(r+\varepsilon)}>0,$$ as the polynomials
$P_{y_{j_0},y_{j_0}}(z)$ and $P_{y_{j_0},x}(z)$ have the same number of roots to the right of $r+\varepsilon,$ by induction hypothesis and by the fact that the leading coefficients are positive. Consider the limit
\begin{eqnarray}P_{x,y_{j_0}}(r)&=&\lim_{\varepsilon\to 0^+} P_{x,x}(r+\varepsilon)  \,{P_{y_{j_0},y_{j_0}}(r+
\varepsilon)\over P_{y_{j_0},x}(r+\varepsilon)}\nonumber\\
&=&P_{y_{j_0},y_{j_0}}(r)
\lim_{\varepsilon\to 0^+} {P_{x,x}(r+\varepsilon) \over P_{y_{j_0},x}(r+\varepsilon)} .\label{lim}
\end{eqnarray}
The polynomials $P_{y,x},$ for $y\in N_x,$ have single roots by induction hypothesis. Thus the limit in the right hand side of (\ref{lim}) is nonzero in view of 
(\ref{vx}).
Since  $P_{y_{j_0},y_{j_0}}(r)\neq 0$ (by induction hypothesis) 
we get that $P_{x,j_0}(r)\neq 0.$
Hence the sign of the limit is determined by the sign of $P_{x,x}(r+\varepsilon).$ 
By plugging $z=r$ into (\ref{rec1}) we get that $P_{x,x'}(r)$ and $P_{x,x}(r+\varepsilon)$ have opposite signs.

Consider now two consecutive roots $r_1<r_2$ of $P_{x,x}(z).$ The signs of
$P_{x,x}(r_1+\varepsilon)$ and $P_{x,x}(r_2+\varepsilon)$ are opposite. Therefore the signs of $P_{x,x'}(r_1)$ and $P_{x,x'}(r_2)$ are also opposite. Thus $P_{x,x'}(z)$ must vanish in the interval $(r_1,r_2).$ 

Assume now that $r$ is the largest root of $P_{x,x}(z).$ Then $P_{x,x}(r+\varepsilon)>0$ for small positive $\varepsilon.$ By the above reasoning we have $P_{x,x'}(r)<0,$ which means that $P_{x,x'}$ must vanish somewhere to the right of $r,$ as the leading coefficient is positive.
Similarly if $r$ is the smallest root of $P_{x,x}(z)$ then the signs of $P_{x,x'}(r)$ and $P_{x,x}(r+\varepsilon)$ are opposite. But since the degree of $P_{x,x'}$ is by one greater than the degree of $P_{x,x}(z)$ and the leading coefficients are positive, we get that $P_{x,x'}$ must vanish below $r.$
\end{proof}
For $x\in \Gamma$ let $J_x$ denote the truncation of the Jacobi matrix $J$ to the subtree $\Gamma_x,$ i.e. the matrix with the parameters $\lambda^x_y, $ $\beta^x_y$ so that
$$\lambda_y^x=\begin{cases}
\lambda_y & {\rm for}\ y\in \Gamma_x\setminus\{x\}\\
0 & {\rm for}\ y\notin \Gamma_x\setminus\{x\}
\end{cases}\hspace{2cm} \beta_y^x=\begin{cases}
\beta_y & {\rm for}\ y\in \Gamma_x\\
0 & {\rm for}\ y\notin \Gamma_x
\end{cases} $$ 
\begin{thm} Let $x\in \Gamma.$ 
Let $r$ belong to the spectrum of $J_x.$ Then $r$ satisfies at least one of the two conditions 
\begin{enumerate}
\item[(a)] $P_{x,x'}(r)=0.$
\item[(b)] There exist $y\in \Gamma_x$ and
$y_1,y_2\in N_y$ so that $P_{y_1,y}(r)=P_{y_2,y}(r)=0.$
\end{enumerate} 
\end{thm}
\begin{proof}
First we will show that the numbers described in the theorem belong to the spectrum of $J_x.$
Assume $P_{x,x'}(r)=0.$ By Theorem 1 we have $P_{x,x}(r)\neq 0.$ By Lemma 2 for any nonreal $z$ there is a solution $v_x$ of the equation $Jv_x=zv_x$ so that $v_x(y)=P_{x,y}(z)$  for $y\in \Gamma_x'.$
Let $$u(y)=\lim_{\varepsilon\to 0} P_{x,y}(r+i\varepsilon),\quad y\in \Gamma_x.$$ Then
$u$ satisfies $J_xu=ru.$ Moreover $u$ is nonzero as $u(x)=P_{x,x}(r)\neq 0.$ 

Assume now that there exist $y\in \Gamma_x$ and $y_1,y_2\in N_y$ so that $P_{y_1,y}(r)=P_{y_2,y}(r)=0.$ By the above reasoning there are two nonzero solutions $u_1,u_2,$ defined on $\Gamma_{y_1},\ \Gamma_{y_2},$ respectively,  of the equations $J_{y_1}u_1=ru_1$ and  $J_{y_2}u_2=ru_2$ and $u_1(y_1)\neq 0,$ $u_2(y_2)\neq 0.$ Consider the function $u_{y_1,y_2}$ defined on $\Gamma_y$ as follows
$$u_{y_1,y_2}(t)=\begin{cases}\ \ \lambda_{y_2} u_2(y_2)u_1(t) & t\in \Gamma_{y_1},\\
-\lambda_{y_1}u_1(y_1)u_2(t) & t\in \Gamma_{y_2},\\
0 & t\notin \Gamma_{y_1}\cup\Gamma_{y_2}.
\end{cases}$$
Then $u_{y_1,y_2}\neq 0$ and $J_xu_{y_1,y_2}=Ju_{y_1,y_2}=ru_{y_1,y_2}.$

Assume that for $y\in \Gamma_x$ and $y_1,y_2,\ldots, y_n\in N_y$ we have
$$P_{y_1,y}(r)=P_{y_2,y}(r)=\ldots = P_{y_n,y}(r)=0.$$
Then the eigenvectors $$u_{y_1,y_2},\ldots, u_{y_1,y_n}$$ are linearly independent, as the support of $u_{y_1,y_i}$ coincides with $\Gamma_{y_1}\cup \Gamma_{y_i}.$
Hence the dimension of the space spanned by these eigenvectors is at least $n-1.$

In the previous part of the proof we have constructed eigenvectors corresponding to the set of numbers described in the theorem. We will calculate the dimension of the space spanned by these eigenvectors. The proof will be complete if the dimension  coincides with the dimension of the space $\ell^2(\Gamma_x),$ i.e. with $\#\Gamma_x.$ 
We will use induction with respect to $\ell(x).$
Assume the conclusion is valid for $l(x)=n.$
Let $\ell(x)=n+1.$ Denote
$N_x=\{y_1,y_2,\ldots, y_k\}.$ Let $n_j={\rm deg} P_{y_j,x}.$ Every eigenvector of $J_{y_j}$ corresponding to the case (b) is an eigenvector of $J_x$ as well. Therefore, by induction hypothesis, the dimension of the linear span of all eigenvectors of $J_{y_i}$ corresponding to the case (b) is equal
$$\#\Gamma_{y_i}-\deg P_{y_i,x}.$$
Such eigenvectors corresponding to $J_{y_i}$ and $J_{y_j},$ for $i\neq j$ have disjont supports, hence the total dimension of the eigenvectors corresponding to the case (b) for 
$J_{y_1},\ldots, J_{y_k}$ is equal 
$$\sum_{j=1}^k\#\Gamma_{y_j}-\sum_{j=1}^k{\rm \deg} P_{y_j,x}
$$  Consider the product
$$P_{y_1,x}(z)\ldots P_{y_k,x}(z).$$
We know that every polynomial $P_{y_j,x}$ has single roots. We have
$$P_{y_1,x}(z)\ldots P_{y_k,x}(z)=c\prod_{l=1}^L(z-r_l)^{n_l}.$$ By the reasoning performed in the first part of the proof, the root $r_l$ gives rise to $n_l-1$ linearly independent eigenvectors of $J_x.$ Moreover the degree of the polynomial $P_{x,x'}$ is equal to $L+1$ as
$$\deg P_{x,x'}=\deg P_{x,x}+1,$$
and $\deg P_{x,x}=L$ (cf. (\ref{vx})).
  The roots of $P_{x,x'}$ lead to $L+1$ linearly independent eigenvectors of $J_x,$ which are linearly independent from the ones constructed in (b), as they do not vanish at $x.$ Summarizing the number of linearly independent eigenvectors of $J_x$ is not less then
$$
\sum_{j=1}^k\#\Gamma_{y_j}-\sum_{j=1}^k{\rm \deg} P_{y_j,x}+\sum_{l=1}^L(n_l-1)+L+1=
\sum_{j=1}^k\#\Gamma_{y_j}+1=\#\Gamma_x.
$$
\end{proof}

\section{Essential selfadjointness and defect indices}
Let $z\notin \mathbb{C}.$ The function $v\in \ell^2(\Gamma)$ belongs to the defect space corresponding to $\overline{z}$ if $v$ is orthogonal to ${\rm Im}\, (\overline{z}I-J)=(\overline{z}I-J)({\mathcal F}(\Gamma)).$ In particular $v$ is orthogonal to $(\overline{z}I-J)\delta_x$ for any $x\in \Gamma.$ This implies  $Jv=zv.$
\begin{prop}
The defect indices of the operator $J$ cannot be greater than~1.
\end{prop}
\begin{proof} Fix a nonreal number $z.$ 
Let $v\in \ell^2(\Gamma)$ satisfy $v\neq 0$  and
$Jv=zv.$
By Corollary 1 the function $v$ is unique up to a constant multiple. 
\end{proof}
Proposition 2 implies
\begin{cor} Let $J$ be a Jacobi matrix on $\Gamma.$ 
Fix a nonreal number $z$ and let $v$ denote the unique, up to a constant multiple, nonzero solution of the equation $Jv=zv.$
Then $J$ is essentially selfadjoint if and only if 
$ v\notin \ell^2(\Gamma).$ 
\end{cor}
\begin{thm}
There exist essentially  nonselfadjoint Jacobi matrices on $\Gamma.$ 
\end{thm}
\begin{proof}  We set $\beta_x\equiv 0.$ Fix a nonreal number $z.$
 Choose an infinite path $x_n$ in $\Gamma$ so that $l(x_n)=n.$  We will construct a matrix $J$ by induction on $n.$ 
 Assume
 we have constructed a matrix $J$ on $\Gamma_{x_{n-1}}\setminus \{x_{n-1}\}$ and a nonvanishing function $v$ on $\Gamma_{x_{n-1}}$ so that
 $$ \|v\mid_{\Gamma_{x_{n-1}}}\|_2^2\le 1-2^{-(n-1)}$$
 and 
 $$ Jv(x)=zv(x),\quad x\in \Gamma_{x_{n-1}}\setminus \{x_{n-1}\}.$$
 We want to extend the definition of $J$ and $v$ so that the conclusion remains valid when $n-1$ is replaced by $n.$
 
 Our first task is to define $\lambda_{x_{n-1}}$ and $v(x_{n})$ so that
 $$ zv(x_{n-1})=\lambda_{x_{n-1}} v(x_n)+\sum_{y\in N_{x_{n-1}}}\lambda_yv(y) ,$$ i.e.
 \begin{equation}\label{nonzero}
 \lambda_{x_{n-1}} v(x_n)=zv(x_{n-1})-\sum_{y\in N_{x_{n-1}}}\lambda_yv(y).
 \end{equation}
  The right hand side of (\ref{nonzero})  cannot vanish by Lemma 2. 
We will define  $\lambda_{x_{n-1}}$ and $v(x_n)$ so as to satisfy (\ref{nonzero}).
 By specifying $\lambda_{x_{n-1}}$ large enough we may assume that
 $$|v(x_n)|^2\le 2^{-n-1}.$$
 
 For any $y\in N_{x_n}$ and $y\neq x_{n-1}$ consider the subtree $\Gamma_y\setminus \{y\}.$ Set $\lambda_x=1$ for any $x\in \Gamma_y\setminus \{y\}.$ By Lemma 3 there is a nonzero solution $v_y$  defined on $\Gamma_y$ satisfying
 $$Jv_y(x)=zv_y(x),\quad x\in \Gamma_y\setminus \{y\}.$$ We may assume
 that
 $$\sum_{y\in N_{x_n}\setminus\{x_{n-1}\}}\|v_y\mid_{\Gamma_y}\|_2^2\le 2^{-n-1}.$$
 We want to define the numbers $\lambda_{y}$ for $y\in N_{x_n}$ and $y\neq x_{n-1}$ so that
 $$zv_y(y)=Jv(y)=\lambda_y v_y(x_n)+\sum_{x\in N_y}\lambda_xv_y(x).$$
 Hence we want to have
 \begin{equation}\label{th}\lambda_y= {zv_y(y)-\displaystyle\sum_{x\in N_y}\lambda_xv_y(x)\over v(x_n)}.
 \end{equation}
 By Lemma 2 the numerator (\ref{th}) cannot vanish. We may multiply $v_y$ by a constant of absolute value 1 so that the expression on the right hand side of (\ref{th}) becomes positive. In this way the values $\lambda_y$ for
  $y\in N_{x_n}$ and $y\neq x_{n-1}$ are defined. 
  We extend the definition of $v$ to $\Gamma_{x_n}$ by setting
  $$
v(x)=v_y(x) , \ x\in \Gamma_y,\ y\neq x_{n-1}.$$
On the way we have also extended the definition of $J$ so that
$$Jv(x)=zv(x),\quad x\in \Gamma_{x_n}\setminus\{x_n\}.$$
Moreover by construction we have
\begin{multline*}\|v\mid_{\Gamma_{x_{n}}}\|_2^2
=\|v\mid_{\Gamma_{x_{n-1}}}\|_2^2+\sum_{y\in N_{x_n}, y\neq x_{n-1}}\|v\mid_{\Gamma_{y}}\|_2^2 +|v(x_n)|^2\\ \le 1-2^{-(n-1)}+2^{-n-1}+2^{-n-1}=1-2^{-n}.
\end{multline*}
 \end{proof}
{\bf Remark 1.} The Jacobi matrix $J$ constructed in the proof satisfies $\beta_x\equiv 0$ and $\lambda_x=1$ for vertices $x$ whose distance from the path $\{x_n\}$ is greater than 2.

{\bf Remark 2.} Another way of proving Theorem 3 is as follows. Fix any Jacobi matrix $J_0$ so that the operator $J_0$ is bounded on $\ell^2(\Gamma).$ For example
Let $\beta_x\equiv 0$ and $\lambda_x=d_x^{-1/2},$ where $d_x=\# N_{x'}.$
Let $S$ denote the operator acting according to the rule
$$Sv(x)=\lambda_xv(x').$$ Then
\begin{multline*}\|Sv\|_2^2=\sum_{x\in \Gamma} |Sv(x)|^2=
\sum_{x\in \Gamma} \lambda_x^2|v(x')|^2\\ =
\sum_{x'\in \Gamma\setminus \Gamma_0}|v(x')|^2\sum_{x\in N_{x'}}\lambda_x^2
=\sum_{x'\in \Gamma\setminus \Gamma_0}|v(x')|^2\le \|v\|_2^2.
\end{multline*}
The operator $S$ is thus bounded. The adjoint operator $S^*$ acts by the rule
\begin{align*}&S^*v(x)=\sum_{y\in N_x} \lambda_yv(y), & x\notin\Gamma_0,\\
&S^*v(x)=0, &x\in \Gamma_0.
\end{align*}
Then $J_0=S+S^*$ is the Jacobi matrix such that $\|J_0\|_{2\to 2}\le 2.$
Fix an infinite path $\{x_n\}$ and a sequence of positive numbers $\{\lambda_n\}.$ Let $J_1$ be the degenerate Jacobi matrix defined
by $\beta_x\equiv 0$ and $\lambda_{x_n}=\lambda_n,$ $\lambda_x=0$ for $x\notin \{x_n\}.$
Choose the coefficients $\lambda_n$ so that the classical Jacobi matrix associated with the coefficients $\lambda_n$ and $\beta_n\equiv 0$ is not essentially selfadjoint. For example let $\lambda_n=2^{-n}.$ Let $J=J_0+J_1.$ The matrix $J$ is nondegenerate. Moreover $J$ is not essentially selfadjoint as a bounded perturbation of not essentially selfadjoint operator.

The next theorem provides a relation between Jacobi matrices on the tree $\Gamma$ and classical Jacobi matrices associated with the infinite paths of $\Gamma.$ 

The Jacobi matrix $J$ on $\Gamma$ will be called {\it symmetric} if $\beta_x\equiv 0.$
\begin{thm} Let $J$ be a nonessentially selfadjoint symmetric Jacobi matrix on $\Gamma.$  Choose an infinite path
$\{x_n\}$ so that $l(x_n)=n.$ Then the classical Jacobi matrix $J_0$ with $\lambda_n=\lambda_{x_n}$ and $\beta_n\equiv 0$ is nonessentialy selfadjoint.
\end{thm}

Before proving the theorem we will need the following lemma.
\begin{lemma} Let $J$ be a  symmetric Jacobi matrix on $\Gamma.$  
Let $Jv=iv$ and $v(x_0)=1$ for a vertex $x_0$ on the level $0.$ Then the function
$\tilde{v}(x) = i^{-l(x)}v(x)$ is positive.
\end{lemma}
\begin{proof}
By assumptions we have
$$
iv(x)=\lambda_{x}v(x')+\sum_{y\in N_x} \lambda_yv(y).
$$
Thus
\begin{equation}
\tilde{v}(x)=\lambda_{x}\tilde{v}(x')-\sum_{y\in N_x} \lambda_y\tilde{v}(y).\label{recu1}
\end{equation}
We know that $\tilde{v}$ cannot vanish and is unique up to a constant multiple. The function ${\rm Re}\,\tilde{v}$ satisfies (\ref{recu1}) and takes the value 1 at $x_0.$ Thus
$\tilde{v}={\rm Re}\,\tilde{v},$ i.e. $\tilde{v}$ is real valued. We will show that $\tilde{v}(x)$ is positive by induction. Observe that if $\tilde{v}(x)$ is positive for any vertex on  level zero then
by (\ref{recu1}) $\tilde{v}$ is positive.
Assume the opposite, i.e. $\tilde{v}$ is  negative at some vertices on  level zero.  Since $\tilde{v}(x_0)=1$ there are two vertices $y_1,y_2$ on level zero so that
$y_1'=y_2'$ and 
$\tilde{v}(y_1)>0,$  $\tilde{v}(y_2)<0.$ By (\ref{recu1}) evaluated at $x=y_1$ and $x=y_2$ we get
that $\tilde{v}(y_1')>0$ and $\tilde{v}(y_2')<0,$ which gives a contradiction.

\end{proof}
\begin{proof}[Proof of Theorem 3]
By (\ref{recu1}) evaluated at $x=x_n$ we obtain
$$
\tilde{v}(x_n)=\lambda_{x_n}\tilde{v}(x_{n+1})-\sum_{y\in N_{x_n}} \lambda_y\tilde{v}(y).
$$
Hence
$$\lambda_{x_n}\tilde{v}(x_{n+1})-\lambda_{x_{n-1}}\tilde{v}(x_{n-1})=\tilde{v}(x_n)+\sum_{y\in N_{x_n}\atop y\neq x_{n-1}} \lambda_y\tilde{v}(y)> 0.$$ The last inequality follows from Lemma 4.
Therefore
\begin{eqnarray*}
\tilde{v}(x_{2n})&\ge &{\lambda_{x_0}\lambda_{x_2}\ldots \lambda_{x_{2n-2}}\over \lambda_{x_1}\lambda_{x_3}\ldots \lambda_{x_{2n-1}}}\tilde{v}(x_0),\\
\tilde{v}(x_{2n+1})&\ge &{\lambda_{x_1}\lambda_{x_3}\ldots \lambda_{x_{2n-1}}\over \lambda_{x_2}\lambda_{x_4}\ldots \lambda_{x_{2n}}}\tilde{v}(x_1).
\end{eqnarray*}
By assumptions the seqeunce $\tilde{v}(x_n)$ is square summable. Thus
\begin{equation}\label{non}
 \sum_{n=1}^\infty \left ({\lambda_{x_0}\lambda_{x_2}\ldots \lambda_{x_{2n-2}}\over \lambda_{x_1}\lambda_{x_3}\ldots \lambda_{x_{2n-1}}}\right )^2 +\left ({\lambda_{x_1}\lambda_{x_3}\ldots \lambda_{x_{2n-1}}\over \lambda_{x_2}\lambda_{x_4}\ldots \lambda_{x_{2n}}} \right )^2<\infty.
 \end{equation}
The last inequality is equivalent to nonessential selfadjointness of the classical Jacobi matrix $J_0$ with $\lambda_n=\lambda_{x_n}$ and $\beta_n\equiv 0.$ Indeed, let $p_n$ and $q_n$ denote the polynomials of the first and the second kind associated with $J_0.$ Then (\ref{non}) reduces to
$$\sum_{n=1}^\infty [p_n^2(0)+q_n^2(0)]<\infty.$$
\end{proof}
{\bf Remark}
The symmetry assumption $\beta_x\equiv 0$ is essential. There exist nonessentially selfadjoint Jacobi matrices $J$ on $\Gamma$ so that the
classical Jacobi matrix $J_0$ associated with the path $\{x_n\}$ is essentially selfadjoint. Indeed for every vertex $x_n,$ $n\ge 1,$ fix a vertex $y_{n-1}\neq x_{n-1}$ in $N_{x_n}.$ Let $P$ denote the orthogonal projection from $\ell^2(\Gamma)$ onto $\ell^2(\{x_n,y_n\}_{n=0}^\infty).$
We will consider   Jacobi matrices $J$ so that $\beta_x=0$ for $x\notin\{y_n\}_{n=0}^\infty.$
Let $J_1=PJP.$ First we are going to study the essential selfadjointness of the operator $J_1.$ To this end consider the equation
$J_1v(x)=zv(x).$ This is equivalent to
\begin{eqnarray*}
zv(x_n)&=&\lambda_{x_n}v(x_{n+1})+\lambda_{x_{n-1}}v(x_{n-1})+\lambda_{y_{n-1}}v(y_{n-1}),\\
zv(y_{n-1})&=&\lambda_{y_{n-1}}v(x_n)+\beta_{y_{n-1}}v(y_{n-1}).
\end{eqnarray*}
We have
\begin{equation}\label{y}
v(y_{n-1})={\lambda_{y_{n-1}}\over z-\beta_{y_{n-1}}}v(x_n).
\end{equation}
Hence
$$zv(x_n)=\lambda_{x_n}v(x_{n+1})+\lambda_{x_{n-1}}v(x_{n-1})+{\lambda_{y_{n-1}}^2\over z-\beta_{y_{n-1}}}v(x_n).$$
Set $z=i,$ $v_n:=v_{x_n},$ $\mu_n:=\lambda_{y_{n-1}},$ $\lambda_n:=\lambda_{x_n}$ and $\beta_n:=\beta_{y_{n-1}}.$
Then
$$\left [{\beta_n\mu_n^2\over 1+\beta_n^2}+\left (1+ {\mu_n^2\over 1+\beta_n^2}\right )\,i\right ]v_n=
\lambda_nv_{n+1}+\lambda_{n-1}v_{n-1}.$$
Set $\mu_n^2=1+\beta_n^2.$ Then we obtain
$$ 2iv_n=\lambda_nv_{n+1}-\beta_nv_n+\lambda_{n-1}v_{n-1}.$$
Assume the classical Jacobi matrix with coefficients $\lambda_n$ and $-\beta_n$ is not essentially selfadjoint. Then
the sequence $v_n$ is square summable. Moreover (\ref{y}) implies 
$$|v(y_{n-1})|^2={\mu_n^2\over 1+\beta_n^2}|v_n|^2=|v_n|^2.$$
Hence
$$\|v\|^2=\sum_{n=0}^\infty |v(x_n)|^2+|v(y_n)|^2<\infty,$$ i.e. the operator $J_1$ is not essentialy selfadjoint.
Let $J_2$ be any bounded Jacobi matrix on $\Gamma.$ Then $J=J_1+J_2$ is a nonessentially selfadjoint Jacobi matrix on $\Gamma.$ 

The matrix $J_0$ is associated with the coefficients $\lambda_n=\lambda_{x_n}$ and $\beta_n\equiv 0.$ 
Thus, in order to conclude the reasoning it suffices to prove the following.
\begin{lemma} There exists a nonselfadjoint classical Jacobi matrix $J$ 
$$Jx_n=\lambda_n x_{n+1}-\beta_nx_n+\lambda_{n-1}x_{n-1}$$
so that the Jacobi matrix
$$J'x_n=\lambda_n x_{n+1}+\lambda_{n-1}x_{n-1}$$
is essentially selfadjoint.
\end{lemma}
\begin{proof} We will assume that $\beta_n\neq 0.$
Nonessential selfadjointness of $J$ is equivalent to the fact that every solution of the recurrence relation
$$0=\lambda_n x_{n+1}-\beta_nx_n+\lambda_{n-1}x_{n-1},\qquad n\ge 1,$$
is square summable.
Assume the sequence $x_n$ satisfies this recurrence relation. Then
\begin{eqnarray}
\beta_{2n}x_{2n}&=&\lambda_{2n}x_{2n+1}+\lambda_{2n-1}x_{2n-1},\label{first}\\
\beta_{2n+1}x_{2n+1}&=&\lambda_{2n+1}x_{2n+2}+\lambda_{2n}x_{2n}.\nonumber
\end{eqnarray}
Thus
\begin{eqnarray}
x_{2n+1}&=&{\lambda_{2n+1}\over \beta_{2n+1}}x_{2n+2}+{\lambda_{2n}\over \beta_{2n+1}}x_{2n},\label{one}\\
x_{2n-1}&=&{\lambda_{2n-1}\over \beta_{2n-1}}x_{2n}+{\lambda_{2n-2}\over \beta_{2n-1}}x_{2n-2}\nonumber
\end{eqnarray}
Plugging in the last two equations into (\ref{first})  results in
$$\left (\beta_{2n}-{\lambda_{2n}^2\over \beta_{2n+1}}-{\lambda_{2n-1}^2\over \beta_{2n-1}}\right )x_{2n}=
{\lambda_{2n}\lambda_{2n+1}\over \beta_{2n+1}}x_{2n+2}+{\lambda_{2n-2}\lambda_{2n-1}\over \beta_{2n-1}}x_{2n-2}.$$
Let $\beta_{2n-1}=a\lambda_{2n-1}$ and
$$\beta_{2n}={\lambda_{2n}^2\over \beta_{2n+1}}+{\lambda_{2n-1}^2\over \beta_{2n-1}}.$$
Then
$$0=\lambda_{2n}x_{2n+2}+\lambda_{2n-2}x_{2n-2}.$$
Choose an increasing sequence $\lambda_{2n}$ so that every solution $u_{2n}$ of the last equation is square summable. 
Assume also that $\lambda_{2n}=\lambda_{2n+1}.$ Then
by (\ref{one}) we get
$$|x_{2n+1}|\le |a||x_{2n+2}|+{\lambda_{2n}\over \lambda_{2n+1}}|x_{2n}|=|a||x_{2n+2}|+|x_{2n}|.$$
Thus the sequence $x_n$ is square summable, i. e. the Jacobi matrix $J$ is not essentially selfadjoint.

On the other hand the Jacobi matrix $J',$ under assumption $\lambda_{2n}=\lambda_{2n+1},$ is essentially selfadjoint.
Indeed, the sequence $x_{2n-1}=0$ and $x_{2n}=(-1)^n$ satisfies
$J'x=0$ and it is not square summable.
\end{proof}
{\bf Remark.} Following the proof it is possible to construct the coefficients $\lambda_n$ and $\beta_n$ explicitly.
Let $\lambda_{2n+1}=\lambda_{2n}=q^n$ for $q>1.$ Then $\beta_{2n+1}=aq^{n}$ and $\beta_{2n}=a^{-1}[q^n+
q^{n-1}].$

The following lemma is straightforward but useful. 
\begin{lemma}\label{lemma} Consider a symmetric operator $A$ on a Hilbert space $\mathcal{H}.$ Let $\mathcal{H}_0$ be a finite dimensional subspace of $D(A)\subset \mathcal{H}$ and let $P_{\mathcal{H}_0}$ denote the orthogonal projection onto $\mathcal{H}_0.$  Define the operator $\tilde{A}:\mathcal{H}_0^\perp\to \mathcal{H}_0^\perp$ by
$$ \tilde{A}=(I-P_{\mathcal{H}_0})A(I-P_{\mathcal{H}_0}).$$
The operator $\tilde{A}$ is essentialy selfadjoint if and only if $A$ is essentially selfadjoint.
\end{lemma}
\begin{thm}
Assume $J$ is not essentially selfadjoint. Fix a vertex $x_0$ on the level zero and a nonreal number $z.$
Let a  function $u(x)$ satisfy $u\neq 0,$ $u(x_0)=0$ and $$Ju(x)=zu(x),\quad x\neq x_0.$$
Then $u$ is square summable on $\Gamma.$
\end{thm}
\begin{proof}
Let $\mathcal{H}_0=\mathbb{C}\delta_{x_0}.$ The operator $\tilde{J}$ acts on $\ell^2(\Gamma\setminus\{x_0\})$ and is not essentially selfadjoint by Lemma \ref{lemma}. Moreover if $\tilde{u}$ denotes the truncation of $u$ to $\tilde{\Gamma}=\Gamma\setminus\{x_0\}$
we have
$$ \tilde{J}\tilde{u}(x)=z\tilde{u}(x),\quad x\in\tilde{\Gamma}.$$
By Lemma 1 we know that $\tilde{u}$ cannot vanish. Since  $\tilde{J}$ is not essentially selfadjoint there exists a function $0\neq\tilde{v}\in \ell^2(\tilde{\Gamma})$ so that
$$ \tilde{J}\tilde{v}(x)=z\tilde{v}(x),\quad x\in\tilde{\Gamma}.$$
By Lemma 3 applied to $\tilde{\Gamma}$ we get that $\tilde{u}(x)=c\tilde{v}(x)$ for $x\in\tilde{\Gamma}.$
\end{proof}
Fix an infinite path $\{x_n\}$ so that $l(x_n)=n.$ By Corollary 2, for a nonreal number $z$ there exist two nonzero solutions $v_z$ and $u_z$ on $\Gamma$ such that 
\begin{align}
&v_z(x_0)=1, \ v_z(x_1)={z-\beta_{x_0}\over \lambda_{x_0}},\quad u_z(x_0)=0,\ u_z(x_1)={1\over \lambda_{x_0}}\\ 
&Jv_z(x)=zv_z(x),\ Ju_z(x)=zu_z(x),\quad x\in \Gamma\setminus\{x_0\}.
\end{align}
Observe that we have $$Jv_z(x)=zv_z(x),\quad  {\rm for}\ x\in \Gamma.$$
The functions $v_z$ and $u_z$ satisfying (6) and (7) will be called  {\it the solution} and  {\it the associated solution} of the equation
$$Jf(x)=zf(x),\quad x\in \Gamma\setminus\{x_0\}.$$
Summarizing we get
\begin{prop} Let $J$ be nonessentially selfadjoint Jacobi matrix on $\Gamma.$ Fix a vertex $x_0$ on the level $0.$
For any nonreal number $z$ every solution of the equation
$$zv(x)=\lambda_xv(x')+\beta_xv(x)+\sum_{y\in N_x}\lambda_yv(y),\qquad x\neq x_0$$
is square summable.
\end{prop}

Consider the graph obtained from $\Gamma$ by removing all links $(x_n,x_{n+1})$ (we do not remove vertices).
This graph splits into the infinite sum of finite subtrees $\Gamma_n.$ The tree $\Gamma_n$ contains the vertex $x_n.$ Moreover
$x_n$ is the only vertex of $\Gamma_n$ on the level $n.$
\begin{lemma}
Let $x\in \Gamma_n,$ for some $n\ge 1.$ Then $v_z(x_{n})u_z(x)=u_z(x_{n})v_z(x).$ 
\end{lemma}
\begin{proof}
By Lemma 1 we know that $v_z$ and $u_z$ cannot vanish. Both functions satisfy $Ju_z(x)=zu_z(x),$ $Jv_z(x)=zv_z(x)$ for $x\in \Gamma_n\setminus\{x_n\}.$ 
By Lemma 3  we get  $v_z(x)=cu_z(x)$ for $x\in \Gamma_n.$ Plugging in $x=x_n$ gives the conclusion.
\end{proof}
\begin{prop}
For the solution $v_z$ and the associated solution $u_z$ we have
$$\begin{vmatrix}
v_z(x_n) & u_z(x_n)\\
v_z(x_{n+1}) &u_z(x_{n+1})
\end{vmatrix} = {1\over \lambda_{x_n}}.
$$
\end{prop}
\begin{proof}
By the recurrence relation (\ref{recur}) we get
\begin{eqnarray*}
\lambda_{x_n}v_z(x_{n+1})&=& zv_z(x_n)-\beta_{x_n}v_z(x_n)-\lambda_{x_{n-1}}v_z(x_{n-1})-\sum_{y\in N_{x_n}\setminus \{x_{n-1}\}} \lambda_yv_z(y),\\
\lambda_{x_n}u_z(x_{n+1})&=& zu_z(x_n)-\beta_{x_n}u_z(x_n)-\lambda_{x_{n-1}}u_z(x_{n-1})-\sum_{y\in N_{x_n}\setminus \{x_{n-1}\}} \lambda_yu_z(y).
\end{eqnarray*} 
On multiplying the equations by $u_z(x_n)$ and $v_z(x_n),$ respectively, subtracting sidewise and making use of Lemma 6 gives
$$\lambda_{x_n}\begin{vmatrix}
v_z(x_n) & u_z(x_n)\\
v_z(x_{n+1}) &u_z(x_{n+1})
\end{vmatrix} = \lambda_{x_{n-1}}\begin{vmatrix}
v_z(x_{n-1}) & u_z(x_{n-1})\\
v_z(x_{n}) &u_z(x_{n})
\end{vmatrix}
$$
The conclusion follows as
$$\lambda_{x_0}\begin{vmatrix}
v_z(x_0) & u_z(x_0)\\
v_z(x_{1}) &u_z(x_{1})
\end{vmatrix}=1.$$
\end{proof}
\begin{thm}
Let $J$ be a Jacobi matrix associated with the coefficients $\lambda_x$ and $\beta_x.$  Let $x_n$ denote any infinite path so that $l(x_n)=n.$ Assume
$$\sum_{n=1}{1\over \lambda_{x_n}} =\infty.$$ Then the operator $J$ is essentially selfadjoint.
\end{thm}
\begin{proof}
The result follows by the standard argument from Proposition 4. If $J$ was not essentially selfadjoint then 
the functions $v$ and $u$ would be square summable, thus the series $\sum \lambda_{x_n}^{-1} $ would be summable.
\end{proof}

\noindent {\bf Remark.} The assumption does not depend on the choice of the infinite path, as any two such paths will meet at a certain vertex.

\section{Nonnegative Jacobi matrices on trees}
We say that a matrix $J$ is positive definite if
$$(Jv,v)\ge 0,\quad v\in {\mathcal F}(\Gamma).$$
The next theorem gives characterization of positive definite Jacobi matrices on $\Gamma.$
\begin{thm} 
\begin{enumerate}
\item[(i)]
 Assume there exists  a positive function $m(x)$ on $\Gamma$ such that 
\begin{equation}\label{positive}
\beta_xm(x)\ge  \lambda_xm(x')+\sum_{y\in N_x} \lambda_ym(y),\qquad x\in \Gamma.
\end{equation} 
Then the matrix $J$ is positive definite
\item[(ii)] If
the matrix $J$ is positive definite there exists a positive function $m(x)$ on $\Gamma$ such that 
\begin{equation}\label{positive1}
\beta_xm(x)= \lambda_xm(x')+\sum_{y\in N_x} \lambda_ym(y),\qquad x\in \Gamma.
\end{equation} 
\end{enumerate}
\end{thm}
\begin{proof}
(i). 
For $x\in \Gamma$ let
$$ \alpha_x=\lambda_x{m(x)\over m(x')},\quad \gamma_x=\lambda_x{m(x')\over m(x)}.$$
Thus, on dividing by $m(x),$ the formula (\ref{positive})  takes the form
\begin{equation}\label{beta}
\beta_x\ge \gamma_x+\sum_{y\in N_x}\alpha_y,\qquad x\in \Gamma.
\end{equation}
We have (see (\ref{s}))
\begin{multline*}
(Jv,v)=((S+S^*+M )v,v)\\ =\sum_{x\in \Gamma}\beta_x|v(x)|^2+2{\rm Re}\,\sum_{x\in \Gamma}\lambda_x\overline{v(x)}v(x')\\ \ge 
\sum_{x\in \Gamma}\beta_x|v(x)|^2-2\sum_{x\in \Gamma}\lambda_x|v(x)|\,|v(x')|\\
=\sum_{x\in \Gamma}\beta_x|v(x)|^2-2\sum_{x\in \Gamma}\sqrt{\alpha_x\gamma_x}|v(x)|\,|v(x')|\\
\ge \sum_{x\in \Gamma}\beta_x|v(x)|^2-\sum_{x\in \Gamma}\gamma_x|v(x)|^2-\sum_{x\in \Gamma}\alpha_x|v(x')|^2\\ =
\sum_{x\in \Gamma}\beta_x|v(x)|^2-\sum_{x\in \Gamma}\gamma_x|v(x)|^2-\sum_{x\in \Gamma}|v(x)|^2\sum_{y\in N_x}\alpha_y \\ =
\sum_{x\in \Gamma}\left (\beta_x-\gamma_x-\sum_{y\in N_x}\alpha_y\right ) |v(x)|^2\ge 0.
\end{multline*}
(ii) 
Consider the operator $U$ acting by the rule
$$Uv(x)=(-1)^{l(x)}v(x).$$ Clearly $U$ is a unitary operator. Let $$\tilde{J}=-U^*JU.$$ 
Then $\tilde{J}$ is a nonpositive definite operator and $$\tilde{J}v(x)=\lambda_xv(x')-\beta_xv(x)+ \sum_{y\in N_x}\lambda_y v(y).$$
Fix an infinite path $x_n$ so that $l(x_n)=n.$ Thus $\Gamma=\displaystyle\bigcup_{n=0}^\infty \Gamma_{x_n}.$ Let $P_n$ denote
the orthogonal projection from $\ell^2(\Gamma)$ onto $\ell^2(\Gamma_{x_n})$
and  $\tilde{J}_n=P_n\tilde{J}P_n.$ Then $\tilde{J}_n$ is a bounded nonpositive linear operator.
Therefore $$-a_nI< \tilde{J}_n\le 0<\textstyle{1\over n} I,$$ for a positive constant $a_n.$ 
 Hence
$$0< \tilde{J}_n+a_nI<\left (a_n+\textstyle{1\over n}\right )I.$$
We have 
\begin{equation}\label{pos}
0< ((\tilde{J}_n+a_nI)\delta_x,\delta_x)=a_n-\beta_x,\qquad x\in \Gamma_{x_n}.
\end{equation}
Observe that
\begin{equation}\label{coef}
(\tilde{J}_n+a_nI)\delta_x=\lambda_x\delta_{x'}+(a_n-\beta_x)\delta_x+\sum_{y\in N_x}\lambda_y\delta_y,\qquad x\in \Gamma_{x_n}\setminus\{x_n\}.
\end{equation}
Let 
\begin{multline*}
f_n:=({\textstyle{1\over n}} I-\tilde{J}_n)^{-1}\delta_{x_0}=[(a_n+{\textstyle{1\over n}})I-(\tilde{J}_n+a_nI)]^{-1}\delta_{x_0}\\ =
\sum_{k=0}^\infty {1\over (a_n+{\textstyle{1\over n}})^{k+1}}(\tilde{J}_n+a_nI)^k\delta_{x_0}.
\end{multline*}
By (\ref{pos}) and (\ref{coef}) 
 the function
$$(\tilde{J}_n+a_nI)^k\delta_{x_0}$$ is nonnegative, and positive on all vertices of $\Gamma_{x_n}$ at distance from $x_0$ less or equal to $k.$  Hence $f_n\ge 0$ and $f_n(x)>0$ for any $x\in \Gamma_{x_n}.$ Moreover
$$\tilde{J}_nf_n=\tilde{J}_n( {\textstyle{1\over n}}I-\tilde{J})_n^{-1}\delta_{x_0}=(\tilde{J}_n-{\textstyle{1\over n}} I)({\textstyle{1\over n}} I-\tilde{J}_n)^{-1}\delta_{x_0}+{\textstyle{1\over n}} f_n={\textstyle{1\over n}} f_n-\delta_{x_0}.$$
This results in
$$\lambda_xf_n(x')-\beta_xf_n(x)+\sum_{y\in N_x}\lambda_y f_n(y)= {\textstyle{1\over n}} f_n(x)-\delta_{x_0}(x),\quad x\in \Gamma_{x_n}\setminus\{x_n\}.$$
Let $$m_n(x)={f_n(x)\over f_n(x_0)}.$$ 
Then $m_n(x_0)=1$ and 
\begin{align}\label{imp}\lambda_xm_n(x')+\sum_{y\in N_x}\lambda_y m_n(y)&= (\beta_x+{\textstyle{1\over n}})m_n(x),\  x\in \Gamma_{x_n}\setminus\{x_0,x_n\},\\
\lambda_{x_0}m_n(x_1)&\le  (\beta_{x_0}+{\textstyle{1\over n}}).\label{imp1}
\end{align}
Observe that for any fixed $t\in \Gamma$ the sequence $m_n(t)$ is bounded. Indeed, assume the opposite. Let $t$ be the vertex closest to $x_0,$ so that $m_n(t)$ is unbounded. Let $s$ be the vertex adjacent to $t,$ so that
$$d(x_0,t)=d(x_0,s)+1.$$  Then applying (\ref{imp}) with $x=s$ implies that the sequence $m_n(s)$ is unbounded, which gives a contradiction.

Observe also that for any fixed $t\in \Gamma$ the sequence $m_n(t)$ cannot accumulate at zero. Indeed, assume the opposite.
 Let $t$ be the vertex closest to $x_0$ so that $m_n(t)$ accumulates at zero. Again let  $s$ be the vertex adjacent to $t,$ so that
$$d(x_0,t)=d(x_0,s)+1.$$  Then applying (\ref{imp}) with $x=t$ implies that the sequence $m_n(s)$ also accumulates at zero, which gives a contradiction.

Consider the sequence of functions $m_n.$   Let $m$ be any pointwise accumulation point of this sequence.   Then $m(x)>0$ and by (\ref{imp}) and (\ref{imp1}) we obtain
\begin{align}\label{prep}\lambda_xm(x')+\sum_{y\in N_x}\lambda_ym(y)
 &=\beta_xm(x),\qquad  x\in\Gamma\setminus\{ x_0\},\\
\label{prep1}\lambda_{x_0}m(x_1)&\le  \beta_{x_0}.
 \end{align}
 
 In order to get the conclusion (i.e. to guarantee equality also in (\ref{prep1})) we have to modify  slightly the function $m(x).$ 
 
Observe that after removing all the edges from the  path $\{x_n\}$ the tree $\Gamma$ splits into the sequence
of disjoint trees $\Gamma_n$ so that $x_n\in \Gamma_n.$ By (\ref{prep}) evaluated at $x=x_n$  we have
$$\lambda_{x_n}m(x_{n+1})+\lambda_{x_{n-1}}m(x_{n-1})+\sum_{y\in N_{x_n}\atop y\neq x_{n-1}}\lambda_ym(y)= \beta_{x_n}m(x_n),\qquad n\ge 1.$$
Let the coefficients $c_n$ be defined by $c_0=0$ and
\begin{equation}\label{cn}\sum_{y\in N_{x_n}\atop y\neq x_{n-1}}\lambda_ym(y)=c_nm(x_n),\qquad n\ge 1.
\end{equation} Thus
$$\lambda_{x_n}m(x_{n+1})+\lambda_{x_{n-1}}m(x_{n-1})= (\beta_{x_n}-c_n)m(x_n), \qquad n\ge 1.$$
This implies $\beta_{x_n}\ge c_n.$
Consider the classical Jacobi matrix defined by
$$J_0u(n)=\lambda_{x_n}u(n+1)+(\beta_{x_n}-c_n)u(n)+\lambda_{x_{n-1}}u(n-1).$$
By Theorem 9(i) the matrix $J_0$ is positive definite. Let $p_n$ denote the orthogonal polynomials
associated with $J_0.$ By the well known result we have $(-1)^np_n(0)>0.$ Set $\tilde{m}(x_n)=(-1)^np_n(0).$
Then 
\begin{equation}\label{lam}\lambda_{x_n}\tilde{m}(x_{n+1})+\lambda_{x_{n-1}}\tilde{m}(x_{n-1})= (\beta_{x_n}-c_n)\tilde{m}(x_n).
\end{equation}
Set also 
\begin{equation}\label{til}\tilde{m}(y)={\tilde{m}(x_n)\over m(x_n)}m(y),\qquad y\in \Gamma_n.
\end{equation}
In view of  (\ref{cn}),  (\ref{lam}) and (\ref{til}) we get
$$\lambda_{x_n}\tilde{m}(x_{n+1})+\sum_{y\in N_{x_n}}\lambda_y\tilde{m}(y)=\beta_{x_n}\tilde{m}(x_n),\qquad n\ge 0.$$
Finally, by (\ref{prep}) and (\ref{til})   we have
$$\lambda_xm(x')+\sum_{y\in N_x}\lambda_ym(y)
 =\beta_xm(x),\quad  x\in \Gamma_n\setminus \{x_n\}, n\ge 1.$$
\end{proof}

\begin{prop}
There exist Jacobi matrices $J$ on trees so that the equation
$Jv=tv$ does not admit nonzero solutions for some real values of $t.$
\end{prop}
\begin{proof} We may admit that $t=0.$  Consider a tree $\Gamma$ with $\#N_x=2$ for every vertex $x,$ $\ell(x)\ge 1.$ Fix an infinite path $x_n,$ so that $\ell(x_n)=n.$ Then $N_{x_n}=\{x_{n-1},y_{n-1}\}$ for $n \ge 1.$ We will define the coefficients $\lambda_x$ and $\beta_x$ on $\Gamma_{y_k}$ in such a way that the operator $J$ restricted to $\ell^2(\Gamma_{y_k}\setminus\{y_k\})$ is positive. For example we may set $\lambda_x=1$ and $\beta_x=4$ for any $x\in \Gamma_{y_k}\setminus\{y_k\}.$ In this way if the function $v$ satisfies 
$Jv(y)=0$ for $y\in \Gamma_{y_k}\setminus\{y_k\},$ then either $v=0$ or $v$ cannot vanish on $\Gamma_{y_k}.$ If $v$ does not vanish on  $\Gamma_{y_k}$ its restriction to  $\Gamma_{y_k}$ is unique up to a constant multiple. Let $\lambda_{y_k}=1$ and set $\beta_{y_k}$ in such a way that $v(y_k')=v(x_{k+1})=0.$ Set also $\lambda_{x_k}=1$ and $\beta_{x_k}=0$ for any $k.$ Thus the matrix $J$ is defined. Assume $Ju=0.$ If $u$ vanishes on every subtree $\Gamma_{y_n}$ then by the recurrence relation $u$ vanishes at every vertex $x_n,$ with $n\ge 1, $ as $y_n'=x_{n+1}.$ Moreover by the recurrence relation evaluated at $x_1$ we obtain $v(x_0)=0,$ i.e. $v=0.$ If $u$ does not vanish on every subtree $\Gamma_{y_n},$ let $n$ be the smallest index for which $u$ does not vanish on $\Gamma_{y_n}.$ 
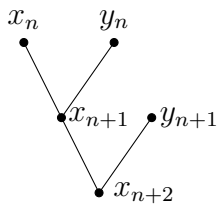
\begin{center}
\begin{tikzpicture}
\draw[fill=black] (0,2) circle (1.5pt);
\draw[fill=black] (1.2,2) circle (1.5pt);
\draw[fill=black] (0.5,1) circle (1.5pt);
\draw[fill=black] (1,0) circle (1.5pt);
\draw[fill=black] (1.7,1) circle (1.5pt);
\node at (0,2.3) {$x_n$};
\node at (1.2,2.3) {$y_n$};
\node at (1,1) {$x_{n+1}$};
\node at (1.6,0) {$x_{n+2}$};
\node at (2.2,1) {$y_{n+1}$};
 \draw (0,2)--(0.5,1);
 \draw (1.2,2)--(0.5,1);
 \draw (1,0)--(0.5,1);
 \draw (1,0)--(1.7,1);
\end{tikzpicture}
\end{center}
Then $v(x_k)=0$ for any $k\le n.$ We must have $v(y_n)\neq 0.$  By construction we also get $v(x_{n+1})=0.$ By the recurrence relation evaluated at $x_{n+1}$ we conclude that $v(x_{n+2})\neq 0.$ This implies that $v$ does not vanish on $\Gamma_{y_{n+1}}.$ 
But by construction $v(x_{n+2})=0,$ which is a contradiction.
\end{proof}

\end{document}